\documentclass[11pt,a4paper]{article}
\usepackage{fullpage}
\usepackage{authblk}

\usepackage{mathrsfs}
\usepackage{amsfonts}
\usepackage{amssymb}
\usepackage{bm}
\usepackage{multirow,bigstrut}
\usepackage{threeparttable}
\usepackage{enumerate}
\usepackage{tikz}
\usepackage{graphicx}
\usepackage{stfloats}
\usepackage{array}
\usepackage{amsmath}
\usepackage[colorlinks,
            linkcolor=red,
            anchorcolor=blue,
            citecolor=green
            ]{hyperref}

\usepackage{algorithm}
\usepackage{algorithmic}
\usepackage{setspace}

\newenvironment{proof}{\noindent{\em \textbf{Proof.}}}{\quad \hfill$\Box$\vspace{2ex}}

\renewcommand{\Re}{\mathcal{R}}
\newcommand{\Ve}{\mathcal{V}}

\newtheorem{theorem}{Theorem}[section]
\newtheorem{definition}[theorem]{Definition}
\newtheorem{remark}[theorem]{Remark}
\newtheorem{corollary}[theorem]{Corollary}
\newtheorem{lemma}[theorem]{Lemma}
\newtheorem{example}[theorem]{Example}

\numberwithin{equation}{section}

\def \H {\mathbb{H}}

\def \C {\mathbb{C}}

\def\qi { \bm {i}}
\def\qj { \bm{j}}
\def\qk { \bm{k}}

\def \bphi {\bm{\phi}}

\def \bfunc {\bm{f}}

\def \bmeta {\bm{\eta}}
\def \bxi {\bm{\xi}}
\def \bp {\bm{p}}
\def \bx {\bm{x}}
\def \by {\bm{y}}

\newcommand{\norm}[1]{\left\lVert#1\right\rVert}
\newcommand{\abs}[1]{\left|#1\right|}

\title{Floquet Theory for Quaternion-valued Differential Equations}
	\author[a]{Dong Cheng\thanks{chengdong720@163.com}}
	
	\author[a]{Kit Ian Kou\thanks{kikou@umac.mo}}
	
	\author[b]{Yong Hui Xia \thanks{xiadoc@163.com}}
	
	\affil[a]{\normalsize{Department of Mathematics, Faculty of Science and Technology, University of Macau, Macao, China}}
	\affil[b]{\normalsize{Department of Mathematics, Zhejiang Normal University, Jinhua, China}}
	
	\date{}

\begin{document}
  \maketitle
\begin{abstract}
	\normalsize

 This paper describes the Floquet theory for quaternion-valued differential equations (QDEs). The Floquet normal form of   fundamental matrix for linear  QDEs  with periodic coefficients is presented and the stability of quaternionic periodic systems is accordingly studied. As an important application of Floquet theory,  we give a discussion on the stability of quaternion-valued Hill's equation. Examples are presented to illustrate  the proposed results.
\end{abstract}


\begin{msc}
34D08, 34B30, 20G20.
\end{msc}

\section{Introduction}\label{S1}

The theory of quaternion-valued differential equations (QDEs) has  gained  a prominent attention  in recent years due to its applications in many fields, including spatial kinematic modelling and attitude dynamics \cite{chou1992quaternion,gupta1998linear}, fluid mechanics \cite{gibbon2002quaternionic,gibbon2006quaternions}, quantum mechanics \cite{alder1986quaternionic,adler1995quaternionic}, etc. A feature of quaternion skew field is  that the multiplication of   quaternion numbers is noncommutative, this property brings challenges to the study of QDEs. Therefore, although QDEs appear in many fields, the mathematical  researches in QDEs are not so many. Leo and Ducati \cite{de2003solving} solved some simple second order quaternionic differential equations by   using the real matrix representation of left/right acting quaternionic operators. Applying the topological degree methods,  Campos  and Mawhin \cite{campos2006periodic} initiated a  study of the $T$-periodic solutions of   quaternion-valued first order differential equations. Later, Wilczynski \cite{wilczynski2009quaternionic,wilczynski2012quaternionic} presented some sufficient conditions for the existence of at least two periodic solutions of the quaternionic Riccati  equation and the existence of at least one periodic solutions of the quaternionic polynomial equations. The existence of periodic orbits, homoclinic loops, invariant tori for 1D autonomous homogeneous QDE $\dot{q}=a q^n, (n=2,3)$ was proposed by Gasull \emph{et al.} \cite{gasull2009one}. The study   of Zhang \cite{zhang2011global}  is devoted to the global struture of   1D quaternion Bernoulli equations. Recently, the  basic theory and fundamental results of linear  QDEs  was established by Kou  and Xia \cite{kouxialinear2018,kou2015linear2,xia2016algorithm}. They proved that the algebraic structure of the solutions to QDEs is different from the classical case. Moreover, for lack of basic theory such as fundamental theorem of algebra, Vieta's formulas of quaternions, it is difficult to solve QDEs. In \cite{kouxialinear2018,kou2015linear2,xia2016algorithm,cheng2018unified}, the authors proposed several new methods to construct the fundamental matrices of linear QDEs.

As a generalization,   QDEs have many properties similar to ODEs. At the same time, for the  relatively complicated algebraic structure of quaternion, one may   encounter various new difficulties when studying QDEs.
\begin{enumerate}
	\item    Factorization theorem and Vieta's formulas (relations between the roots and  the coefficients) for  quaternionic polynomials are not valid (see e. g. \cite{eilenberg1944fundamental, serodio2001zeros,pogorui2004structure}).
	\item  A quaternion matrix usually has infinite number of eigenvalues. Besides, the set of all eigenvectors corresponding to a non-real eigenvalue is not a  module (see e. g. \cite{zhang1997quaternions,rodman2014topics}).
	\item The study of quaternion matrix  equations is of intricacy  (see e. g. \cite{wang2008common,wang2009ranks}).
	\item Even the quaternionic polynomials are not "regular" (an analogue concept of holomorphic). This fact leads to noticeable  difficulties for studying  analytical properties of quaternion-valued functions (see e. g. \cite{sudbery1979quaternionic,wilczynski2009quaternionic}).
\end{enumerate}

Up to present, the theory of QDEs  remains far from systemic. To the best of authors'  knowledge, there was virtually nonexistent   study about the stability theory of QDEs.  Based on this fact, we are motivated to investigate  the stability of the linear QDEs
\begin{equation}\label{homo linear systems}
\dot{x}=A(t)x
\end{equation}
where $A$ is a smooth $n\times  n$ quaternion-matrix-valued function. In particular, we will focus on the important special cases where $A$ is a quaternionic constant or periodic quaternion-valued function. In the real-valued systems, the  well-known Floquet theory indicates that the case where $A$ is  a  periodic  matrix-valued function   is reducible to the constant case (see e. g. \cite{chicone2006ordinary,hale2009ordinary}). Floquet theory is   an effective tool for  analyzing  the periodic solutions and  the stability of dynamic systems. Owing to its importance, Floquet theory has been extended in different directions. Johnson \cite{johnson1980floquet} generalized the Floquet theory to the almost-periodic systems. In \cite{chow1994floquet,kuchment1993floquet,kuchment1994behavior}, the authors extended  the Floquet theory   to  the partial differential equations. Recently, the Floquet theory has been extensively explored  for dynamic systems on time scales (see e. g. \cite{ahlbrandt2003floquet,dacunha2011unified,agarwal2014floquet,adivar2016floquet}).

As a continuation of   \cite{kouxialinear2018,kou2015linear2,xia2016algorithm}, we  generalize the Floquet theory to QDEs in this paper. Specifically, the contributions of this paper are summarized as follows.
\begin{enumerate}
	\item  We show that the stability of  constant coefficient homogeneous linear  QDEs is determined  by   the standard eigenvalues  of its coefficient matrix.
	\item    Floquet normal form of the fundamental matrix for linear  QDEs  with periodic coefficients is presented.
	\item  The monodromy matrix, characteristic multiplier and characteristic exponent for QDEs are defined. Moreover, the stability of quaternionic periodic systems is discussed.
	\item We propose some sufficient conditions for the existence of periodic solution of quaternionic periodic systems.
	\item Without question,  there are some results of ODEs are inevitably invalid for QDEs. We will discuss some of these results. Specifically, we will discuss the stability of quaternion-valued Hill's equation.
\end{enumerate}

The rest of the paper is organized as follow.  In Section \ref{S2}, some basic concepts of quaternion algebra  are reviewed. Besides, several lemmas of quaternion matrices are derived.   Section \ref{S3} is devoted to the stability of  constant coefficient linear  homogeneous QDEs.  In Section \ref{S4}, we establish the Floquet theory for QDEs. Specifically, Floquet normal form of the fundamental matrix for  quaternionic periodic systems is presented. Some important concepts  such as  monodromy matrix, characteristic multiplier and characteristic exponent for QDEs are defined and the stability of quaternionic periodic systems is accordingly studied.   The stability of quaternion-valued Hill's equation is discussed in Section \ref{S5}. Finally,  conclusions are drawn at the end of the paper.

\section{Preliminaries}\label{S2}

%%%%%%%%%%%%%%%%%%%%%%%%%%%%%%%%%%%%%%%%%%%%%%%%%%%%%%%%%%%%%%%%%%%%
\subsection{Quaternion  algebra}\label{S2.1}

The quaternions  were first described by Hamilton  in 1843   \cite{sudbery1979quaternionic}. The algebra of quaternions is   denoted by
\begin{equation*}
\H:= \{q=q_0+q_1\qi  +q_2 \qj+ q_3\qk\}
\end{equation*}
where $q_0,q_1,q_2,q_3$ are real numbers and the elements $\qi$, $\qj$ and $\qk$ obey Hamilton's multiplication rules:
\begin{equation*}
\qi\qj=-\qj\qi=\qk,~~\qj\qk=-\qk\qj=\qi,~~\qk\qi=-\qi\qk=\qj,~~\qi^2=\qj^2=\qi\qj\qk=-1.
\end{equation*}
For every quaternion $q=q_0+\qi q_1+\qj q_2+\qk q_3$, the scalar and vector parts of $q$, are  defined as $\Re(q)=q_0$ and $\Ve(q)=q_1\qi  +q_2 \qj+ q_3\qk$, respectively. If $q= \Ve(q)$, then $q$ is called pure imaginary quaternion.
The quaternion conjugate is defined by $\overline{q}= q_0-\qi q_1-\qj q_2-\qk q_3$, and the norm $|q|$ of $q$ defined as
$|q|^2={q\overline{q}}={\overline{q}q}=\sum_{m=0}^{m=3}{q_m^2}$.
Using the conjugate and norm of $q$, one can define the inverse of $q\in\H\backslash\{0\}$ by $q^{-1}=\overline{q}/|q|^2$.
For each fixed unit pure imaginary quaternion $\bm \varsigma$, the quaternion has subset $\mathbb{C}_{\bm \varsigma}:=\{a+b \bm \varsigma  :a,b\in\mathbb{R}\}$. The  complex number field $\mathbb{C}$ can be viewed as a subset of $\H$ since  it is   isomorphic to $\mathbb{C}_{\qi}$.  Therefore we will denote $\mathbb{C}_{\qi}$ by $\mathbb{C}$ for simplicity.

\subsection{Matrices of quaternions}\label{S2.2}

The quaternion exponential function $ \exp(A)$ for $A\in \H^{n\times n}$ is defined by means of an infinite series as
\begin{equation*}
\exp({A}):=\sum_{n=0}^\infty \frac{A^n}{n!}.
\end{equation*}
When $n=1$ and $A=q\in \H$, analogous to the complex case one may derive a closed-form representation:
\begin{equation*}
e^{q}=\exp(q)= e^{q_0}\left(\cos| \Ve(q)|+\frac{ \Ve(q)}{| \Ve(q)|}\sin| \Ve(q)|\right).
\end{equation*}

Every   quaternion
matrix $A\in\H^{m\times n}$
can be expressed uniquely in the form of
\begin{equation*}
A=A_1+ A_2\qj, ~~~\text{where}~~  A_1,   A_2 \in \C^{m\times n}.
\end{equation*}
Then the \emph{complex adjoint matrix}  \cite{aslaksen1996quaternionic,zhang1997quaternions}  of the quaternion matrix $A$ is defined as
\begin{equation}\label{complex adjoint matrix}
\chi_A=\begin{pmatrix}
A_1&A_2\\
-\overline{A_2} & \overline{A_1}
\end{pmatrix}.
\end{equation}
By using the complex adjoint matrix, the $q$-determinant of $A$ is defined by
\begin{equation}\label{q-det}
\abs{A}_q:=\abs{\chi_A},
\end{equation}
where $\abs{\cdot}$ is the conventional determinant for complex matrices.
By direct computations, it is easy to see that $\abs{A}_q=\abs{A}^2$ when $A$ is a complex matrix.

From \cite{kouxialinear2018}, we know that $\H^n$ over the division ring $\H$ is a right $\H$-module (a similar concept to linear space) and $\bmeta_1,\bmeta_2,\cdot\cdot\cdot,\bmeta_k\in \H^n$ are right linearly independent if
\begin{equation*}
\bmeta_1\alpha_1+\bmeta_2\alpha_2+\cdot\cdot\cdot+\bmeta_k\alpha_k=0,\alpha_i\in \H~~ \text{implies that} ~~\alpha_1=\alpha_2=\cdot\cdot\cdot=\alpha_k=0.
\end{equation*}

Let $A\in\H^{n\times n}$, a   nonzero $\bmeta\in \H^{n\times 1}$ is said to be a \emph{right eigenvector} of $A$ corresponding to the \emph{right eigenvalue} $\lambda\in \H$  provided  that
\begin{equation*}
A\bmeta=\bmeta \lambda
\end{equation*}
holds. A matrix $A_1$ is said to be similar to a matrix $A_2$ if $A_2=S^{-1}AS$ for some nonsingular matrix $S$. In particular, we say that  two quaternions $p,q$ are similar if $p=\alpha^{-1}q\alpha$ for some nonzero quaternion $\alpha$. By Theorem 2.2 in \cite{zhang1997quaternions}, we know that the similarity of quaternions defines an equivalence relation. The set 
\begin{equation*}
[q]:= \{ p= \alpha^{-1}q\alpha: \alpha = \mathbb{H}\setminus \{0\}\}
\end{equation*} 
is called an equivalence class of $q$. It is easy to see that $[q]$ can also be recognized by
\begin{equation*}
[q]:= \{ p\in \mathbb{H}: \Re(p)=\Re(q), \abs{\Ve(p)}=\abs{\Ve(q)}\}.
\end{equation*} 
It follows that any equivalence class $[q]$ has one and only one complex-valued element with nonnegative imaginary part.

 We recall some basic results about   quaternion matrices which can be found, for instance, in \cite{zhang1997quaternions,baker1999right,rodman2014topics}.

\begin{theorem}\label{thm of q matrix}
	Let $A\in\H^{n\times n}$, then the following statements hold.
	\begin{enumerate}
		\item $A$  has exactly $n$ right eigenvalues (including multiplicity) which are  complex numbers with nonnegative imaginary parts. These eigenvalues are called standard eigenvalues of $A$.
		\item If $A$ is a complex matrix and its eigenvalues are  $\lambda_1=\alpha_1+\qi \beta_1, \lambda_2=\alpha_2+\qi \beta_2, \cdots,\lambda_n=\alpha_n+\qi \beta_n$ (repeated according to their multiplicity). Then the standard eigenvalues of $A$ are $\widetilde{\lambda}_1=\alpha_1+\qi \abs{\beta_1}, \widetilde{\lambda}_2=\alpha_2+\qi \abs{\beta_2}, \cdots,  \widetilde{\lambda}_n=\alpha_n+\qi \abs{\beta_n}$. In particular, $\abs{\widetilde{\lambda}_j}=\abs{{\lambda}_j}$ for $j=1,2,\cdots,n$.
		\item $A$ is invertible if and only if $\chi_A$ is invertible.
		%   \item  $\det \chi_A \geq 0$, and the characteristic polynomial of $ \chi_A$ has real coefficients.
		%  \item  Let $\bmeta_1,\bmeta_2,\cdot\cdot\cdot,\bmeta_k$  be eigenvectors of $A$ that correspond to eigenvalues  $\lambda_1,\lambda_2,\cdot\cdot\cdot, \lambda_k$, respectively. If these  eigenvalues   are pairwise non-similar. Then $\bmeta_1,\bmeta_2,\cdot\cdot\cdot,\bmeta_k$ are right linearly  independent.
		\item  If $A$ is (upper or lower) triangular, then the only eigenvalues are the diagonal elements (and the quaternions  similar to them).
	\end{enumerate}
\end{theorem}

Let $\Omega$ be the totality of all $2n\times 2n$ partitioned complex matrices which have form of (\ref{complex adjoint matrix}). It has been shown in  \cite{zhang2001jordan,rodman2014topics}  that $\Omega$ is closed under addition,  multiplication  and inversion. Furthermore, each $A\in \H^{n\times n}$ has a Jordan form in $\C^{n\times n}$.
\begin{lemma}\label{expmultiply}
	\cite{zhang2001jordan}  Let $A,B\in \H^{n\times n}$. Then $\chi_A+\chi_B =\chi_{A+B}\in \Omega$ and  $\chi_A\chi_B = \chi_{AB} \in \Omega$. Moreover, if $A$ is invertible, then $\chi_A^{-1}=\chi_{A^{-1}} \in \Omega$.
\end{lemma}

\begin{lemma}\label{Jordanform}
	\cite{zhang2001jordan,rodman2014topics}  Let $A \in \H^{n\times n}$. Then there exists a $P\in \H^{n\times n}$ such that
	\begin{equation*}
	\chi_P^{-1}\chi_A\chi_P=\begin{pmatrix}
	J&0\\
	0 & \overline{J}
	\end{pmatrix}
	\end{equation*}
	is a Jordan canonical form of $\chi_A$, where and $J\in\C^{n\times n}$ has all its diagonal entries with nonnegative imaginary parts. Consequently, $P^{-1}A P=J$ is a Jordan canonical form of $A$ in $\C^{n\times n}$.
\end{lemma}
\begin{remark}
	The  diagonal entries of $J$ are actually the standard eigenvalues of $A$.
\end{remark}

If $\lambda$ is a standard eigenvalue of $A \in \H^{n\times n}$, its algebraic multiplicity is defined by the number of its occurrences in the Jordan canonical form $J$.  Since the totality of solutions  for
$A\bmeta=\bmeta \lambda$ is not a $\H$-module. Thus we could not use dimensionality of 'eigenspace' to define the geometric multiplicity for $\lambda$. Note that $\lambda$ is a eigenvalue of $\chi_A$ and motivated by Lemma \ref{Jordanform}, we may define the  geometric multiplicity for the standard eigenvalues of quaternion matrices as follows.
\begin{definition}
	Let $\lambda$ be a standard eigenvalue of $A \in \H^{n\times n}$, the geometric multiplicity for $\lambda$ is defined as the  dimensionality of the  (complex) linear space   $\{\bx\in \C^n : (J-\lambda I)\bx=0\}$, where $J$ is the Jordan canonical form of $A$ in $\C^{n\times n}$.
\end{definition}
 %From Theorem \ref{thm of q matrix} and Lemma \ref{Jordanform}, we conclude that if $A\in \C^{n\times n}$ then %its Jordan forms in the complex field $\C$ and in the quaternion skew field $\H$ are  slightly  different.

Employing above lemmas, it is not difficult to verify that  $\Omega$ is also closed under exponential.
\begin{lemma}\label{exp q-matrix eqn}
	Let $A,C\in \H^{n\times n}$, where $C$ is invertible. Then  $e^{\chi_A}=\chi_{e^A}\in \Omega$ and there is a $B \in \H^{n\times n}$ such that $e^B=C$.
\end{lemma}
\begin{proof}
	By Lemma \ref{Jordanform}, there is a $P\in \H^{n\times n}$ such that $P^{-1}A P=J \in\C^{n\times n}$. Observe that $\exp({\overline{J}})=\overline{\exp(J)}$ and therefore
	\begin{equation*}
	\begin{split}
	\chi_P^{-1}e^{\chi_A}\chi_P & =e^{\chi_P^{-1}\chi_A\chi_P}\\
	& =e^{\begin{pmatrix}
		J&0\\
		0 & \overline{J}
		\end{pmatrix}}=\begin{pmatrix}
	e^J&0\\
	0 & e^{\overline{J}}
	\end{pmatrix}=\chi_{e^J}.
	\end{split}
	\end{equation*}
	Hence $e^{\chi_A}=\chi_P\chi_{e^J}\chi_P^{-1}=\chi_{Pe^JP^{-1}}=\chi_{e^{PJP^{-1}}}=\chi_{e^A}$.
	
	For quaternion matrix $C$, there is a $S \in \H^{n\times n}$ such that $S^{-1}CS=K\in \C^{n\times n}$. Since $C$ is  invertible, then $K$ is nonsingular. Moreover, there exists a complex matrix $D$ such that $K=e^D$ by Theorem 2.82 in \cite{chicone2006ordinary}. Therefore
	\begin{equation*}
	\begin{split}
	\chi_C & =\chi_S\chi_K\chi_S^{-1} \\
	& =\chi_S\chi_{e^D}\chi_S^{-1}=\chi_S e^{\chi_D}\chi_S^{-1}=e^{\chi_S \chi_D \chi_S^{-1}}=e^{\chi_{SDS^{-1}}}=\chi_{e^{SDS^{-1}}}.
	\end{split}
	\end{equation*}
	Thus
	$C=e^{SDS^{-1}}$. Set $B=SDS^{-1}$, we complete the proof.
\end{proof}

By Lemma \ref{exp q-matrix eqn} and Theorem \ref{thm of q matrix}, we obtain the following spectral mapping theorem.

\begin{theorem}\label{spectral exponetial map}
	If $A\in\H^{n\times n}$ and $\lambda_1,\lambda_2,\cdots,\lambda_n$ are the standard eigenvalues of $A$ repeated according to their multiplicity, then $e^{\widetilde{\lambda}_1}, e^{\widetilde{\lambda}_2}, \cdots, e^{\widetilde{\lambda}_n}$ are the standard eigenvalues of $e^A$, where $\widetilde{\lambda}_j$ ($j=1,2,\cdots,n$) is defined by
	\begin{equation*}
	\widetilde{\lambda}_j:=
	\begin{cases}
	\lambda_j,& \mathrm{if}~ e^{\lambda_j} ~\mathrm{has~ nonnegative ~imaginary~ part};\\
	\overline{\lambda_j}, &\mathrm{otherwise}.
	\end{cases}
	\end{equation*}
\end{theorem}
\begin{proof}
	If   $\lambda_1,\lambda_2,\cdots,\lambda_n$ are the standard eigenvalues of $A$, then
	$\lambda_1,\lambda_2,\cdots,\lambda_n,\overline{\lambda_1},\overline{\lambda_2},\cdots,\overline{\lambda_n}$ are the eigenvalues of $\chi_A$. By  spectral mapping theorem of complex-valued matrix, we conclude that $\sigma=\{e^{\lambda_1},e^{\lambda_2},\cdots,e^{\lambda_n},e^{\overline{\lambda_1}},e^{\overline{\lambda_2}},\cdots,e^{\overline{\lambda_n}}\}$ is the spectrum of $e^{\chi_A}$. Note that $e^{\chi_A}=\chi_{e^A}$ by Lemma
	\ref{exp q-matrix eqn}, we know that
	$\sigma$ is the spectrum of $\chi_{e^A}$. From Theorem \ref{thm of q matrix}, all elements of $\sigma$ are complex-valued eigenvalues of  $e^A$; in particular, the complex numbers possessing the nonnegative imaginary parts in $\sigma$  are   the standard eigenvalues of $e^A$.
\end{proof}

%%%%%%%%%%%%%%%%%%%%%%%%%%%%%%%%%%%%%%%%%%%%%%%%%%%%%%%%%%%%%%%%%%%%%%%%%%%%%%%%%%%%%%%%%%

\section{Stability of linear   homogeneous  QDEs with constant coefficients}\label{S3}

Analogous to ODEs, we can define the concept of stability (in Lyapunov sense) for QDEs.
\begin{definition}
	Let $\bfunc: [t_0,\infty)\times\H ^n\to \H ^n$. Consider $\dot{\bx}=f(t,\bx), ~t\in [t_0,\infty) $. The solution $\bphi(t,t_0,\bx_0)$ (satisfying initial condition $\bx(t_0)=\bx_0$) is called stable if for any $\epsilon>0$, there is a $\delta>0$ such that $\norm{\bx-\bx_0}<\delta$ implies $\norm{\bphi(t,t_0,\bx)-\bphi(t,t_0,\bx_0)}<\varepsilon$ for all $t\geq t_0$. The solution $\bphi(t,t_0,\bx_0)$   is called asymptotically stable if there is a $\delta>0$ such that  $\lim_{t\to \infty}\norm{\bphi(t,t_0,\bx)-\bphi(t,t_0,\bx_0)}=0$ whenever $\norm{\bx-\bx_0}<\delta$.
\end{definition}

For any $A=(a_{ij})_{n\times n}\in\H^{n\times n}$ and $\bmeta=(\eta_1,\eta_2,\cdots,\eta_n)^T\in \H^n$, the norm of $A$ and $\bmeta$ are respectively  defined by
\begin{equation*}
\norm{A}=\sum_{i,j=1}^n | a_{ij}|,~~~\norm{\bmeta}=\sum_{k=1}^n |\eta_k|.
\end{equation*}
The norm $\norm{\cdot}$ defined for $A$ is a matrix norm. It is easy to verify that for any $A,B\in \H^{n\times n}$, the submultiplicativity  holds, that is
\begin{equation*}
\norm{AB}\leq \norm{A} \norm{B}.
\end{equation*}

By similar arguments to Theorem 1.1 in \cite{afanasiev2013mathematical}, we see that the stability of zero solution of (\ref{homo linear systems}) implies the stability of any other solutions. Thus it is permissible to simply say that system  (\ref{homo linear systems}) is stable (or unstable).

\begin{theorem}\label{judging theorem}
	Let $M(t)$ be a fundamental matrix of (\ref{homo linear systems}). Then the system
	(\ref{homo linear systems}) is stable if and only if  $\norm{M(t)}$ is bounded . The system (\ref{homo linear systems}) is asymptotically stable if and only if $\lim_{t\to \infty}\norm{M(t)}=0$.
\end{theorem}
\begin{proof}
	Let $L$ be an upper bound for $\norm{M(t)}$,  $L_1=\norm{M^{-1}(t_0)}$  and $\bphi(t,t_0,\bxi)$ be the solution of (\ref{homo linear systems}) with $\bphi(t_0,t_0,\bxi)=\bxi=(\xi_1,\xi_2,\cdots,\xi_n)^T$. Then  $\bphi(t,t_0,\bxi)=M(t)M^{-1}(t_0) \bxi$.
	For any $\epsilon>0$, let $\delta=\frac{\epsilon}{L L_1}$, then
	$\norm{\bphi(t,t_0,\bxi)-0}=\norm{M(t)M^{-1}(t_0) \bxi}\leq L L_1 \norm{\bxi}<\epsilon$ whenever $\norm{\bxi}<\delta$.
	%Let $M(t)M^{-1}(t_0)=(\bmeta_1,\bmeta_2,\cdots,\bmeta_n)$.
	If for any $\epsilon > 0$ there is a $\delta>0$ such that
	$\norm{M(t)M^{-1}(t_0) \bxi }<\epsilon$ for $\norm{\bxi} <\delta$. Then
	\begin{equation*}
	\begin{split}
	\norm{M(t)M^{-1}(t_0)} &  =n \norm{M(t)M^{-1}(t_0)(\frac{1}{n}, \frac{1}{n}, \cdots, \frac{1}{n})^T} \\
	&\leq n  \sup_{\norm{\bmeta}\leq  1} \norm{M(t)M^{-1}(t_0)\bmeta}\\
	&= n \sup_{\norm{\bxi}\leq  \delta} \norm{M(t)M^{-1}(t_0) \delta^{-1} \bxi}\\
	&< n \epsilon \delta^{-1}
	\end{split}
	\end{equation*}
	Therefore $\norm{M(t)}< n \epsilon \delta^{-1} L_1^{-1}$ is bounded.
	
	If $\lim_{t\to \infty}\norm{M(t)}=0$. Then $\norm{\bphi(t,t_0,\bxi)-0}=\norm{M(t)M^{-1}(t_0) \bxi}=\norm{M(t)} L_1 \norm{\bxi}$ tends to $0$ as $t\to \infty$ whenever $\norm{\bxi}<\delta$.  Conversely, it is easy to see that if the zero solution is   asymptotically stable, then $\norm{M(t)}$ has to be convergent to $0$ as $t\to \infty$.
\end{proof}

%\begin{theorem}
%   \cite{kou2015linear2} If $\lambda_1, \lambda_2, \cdots, \lambda_k$ are distinct standard eigenvalues of $A\in \H^{n\times n}$. Then the solution of the initial value problem
%   \begin{equation*}
%      \dot{\bx}=A\bx,~~~\bx(t_0)=\bmeta
%   \end{equation*}
%   is given by
%\begin{equation*}
%   \begin{split}
%     \bx(t) & =e^{tA}\bmeta \\
%       & =\sum_{j=1}^{k}\sum_{i=1}^{n_j}\sum_{l=1}^{m_{ji}}\left( \bv_l^{ji}+t \bv_{l-1}^{ji}+\frac{t^2}{2}\bv_{l-2}^{ji}+ \cdots + \frac{t^{l-1}}{(l-1)!} \bv_1^{ji} \right)e^{\lambda_j t} r_l^{ji}
%   \end{split}
%\end{equation*}
%\end{theorem}

By using the Jordan canonical form of $A \in \H^{n\times n}$, we can obtain a matrix representation for $e^{tA}$. Let $P$ be a quaternion matrix such that $P^{-1}AP=J \in C^{n\times n}$, then
$P^{-1}e^{tA} P = e^{tP^{-1}AP}=e^{tJ}$. Let  $\lambda_1, \lambda_2, \cdots, \lambda_k$ be the distinct standard eigenvalues of $A$   that correspond to multiplicities $n_1, n_2, \cdots, n_k$, respectively. Then $J=\mathrm{diag}(J_1, J_2, \cdots, J_k)$ where $J_i =\lambda_i I + N_i$ with $N_i^{n_i}=0$. Thus we have that
\begin{equation*}
e^{tJ_i}=e^{t(\lambda_i I+N_i)}=e^{t\lambda_i}e^{tN_i}=e^{t\lambda_1}\left(I+t N_i+\frac{t^2}{2!}N_i^2+ \cdots + \frac{t^{n_i-1}}{(n_i-1)!} N_i^{n_i-1}\right).
\end{equation*}
Note that $e^{tJ}=\mathrm{diag}(e^{t J_1}, e^{t J_2}, \cdots, e^{t J_k})$, then we obtain an explicit matrix representation for $e^{tA}=Pe^{tJ}P^{-1}$. Moreover, this representation has a similar form with the cases where  $A$ is a real or complex matrix. Hence by   similar arguments to Theorem 4.2 in \cite{hale2009ordinary}, we have the following theorem.
\begin{theorem}\label{stablility of constant system}
	The system $ \dot{\bx}=A\bx$
	\begin{enumerate}
		\item is stable if and only if the standard eigenvalues of $A$ all have non-positive real parts and the algebraic multiplicity equals the geometric multiplicity of each standard eigenvalue with   zero real part;
		\item  is asymptotically stable if and only if all the standard eigenvalues of $A$ have negative real  parts.
	\end{enumerate}
\end{theorem}
\begin{remark}
	Since any two similar quaternions possess the same scalar part, thus the phrase "standard" in Theorem \ref{stablility of constant system} can be removed.
\end{remark}
\begin{example}\label{ex1}
	Consider the system $\dot{\bx}=A\bx$, where
	\begin{equation*}
	A=\begin{pmatrix}
	\qi&\qj&\qj\\
	\qk&1&\qk\\
	0&0&1
	\end{pmatrix}
	\end{equation*}
	The principal fundamental matrix at  $t=0$ ($M(0)=I$) is given by
	\begin{equation*}\label{fundamental matrix 1}
	M(t) = \begin{pmatrix}
	\frac{1-\qi}{2}+ \frac{1+\qi}{2}\gamma_1& \frac{\qk-\qj}{2}+\gamma_2&\qj \gamma_3+\gamma_4-e^t\\
	\frac{\qj-\qk}{2}+\frac{\qk-\qj}{2}\gamma_1 &\frac{1-\qi}{2}-\qj\gamma_2&\qi \gamma_3-\qj\gamma_4-(1-\qj-\qk)e^t\\
	0&0&e^{ t}
	\end{pmatrix},
	\end{equation*}
	where $\gamma_1=e^{(1+\qi)t}$, $\gamma_2= \frac{\qj-\qk}{2}e^{(1-\qi)t}$, $\gamma_3=\frac{\qk-1-\qi-\qj}{2} $,  $\gamma_4=e^{(1+\qi)t}\frac{1-\qi+\qj-\qk}{2}$. By straightforward computations,  we have the   result shown in Table \ref{table-ex-1}.
	\begin{table}[ht]
		\centering
		\tabcolsep7.5pt
		\caption{Description of Example \ref{ex1}}{
			\begin{tabular}{ |@{}c | c | c@{}|}
				\hline
				Fundamental &  The standard  & \multirow{2}*{\centering Stability~}\\
				matrix  & eigenvalues of $A$   &    \\ \hline
				\multirow{3}*{\centering $~\displaystyle\lim_{t\to \infty}\norm{M(t)}=\infty$ } &  $\lambda_1=0$, $\Re(\lambda_1)=0$; & \multirow{3}*{\centering unstable~} \\
				&  $\lambda_2=1$, $\Re(\lambda_2)>0$; &   \\
				&  $\lambda_3=1+\qi$, $\Re(\lambda_3)>0$ &   \\ \hline
		\end{tabular}}
		\label{table-ex-1}
	\end{table}
\end{example}

\begin{example}\label{ex2}
	Consider the system $\dot{\bx}=A\bx$, where
	\begin{equation*}
	A= \begin{pmatrix}
	\qi&1&0\\
	0&\qj&0\\
	0&1&\qk
	\end{pmatrix}
	\end{equation*}
	The principal fundamental matrix at $t=0$ is given by
	\begin{equation*}\label{fundamental matrix 2}
	M(t) = \begin{pmatrix}
	e^{\qi t}& \frac{t}{2} \left(e^{\qi t}-\qk e^{-\qi t} \right)+\frac{1+\qk}{2}\sin t&0\\
	0&e^{\qj t}&0\\
	0& \frac{t}{2} \left(e^{\qj t}+\qi e^{\qj t} \right)+\frac{1-\qi}{2}\sin t&e^{\qk t}
	\end{pmatrix}.
	\end{equation*}
	By straightforward computations,  we have the   result shown in Table \ref{table-ex-2}.
		\begin{table}[ht]
		\centering
		\caption{Description of Example \ref{ex2}}{
			\begin{tabular}{ |@{} c | c | c@{}| }
				\hline
				~Fundamental &  The standard  & \multirow{2}*{\centering Stability~}\\
				matrix  & eigenvalues of $A$   &    \\ \hline
				$\displaystyle \norm{M(t)}$ is   &  $\lambda_1=\lambda_2=\lambda_3=\qi$; & \multirow{2}*{\centering unstable~} \\
				unbounded    &   $\Re(\lambda_1)=0$; &   \\
				\hline
		\end{tabular}}
		\label{table-ex-2}
	\end{table}
	 Notice that the the standard eigenvalue $\lambda=\qi$ has zero real part,  we need to show  its algebraic multiplicity is less than its  algebraic multiplicity  $3$. By some basic calculations, we find a quaternion matrix
		\begin{equation*}
	P= \begin{pmatrix}
	-1+\qi&-2 \qi&-\qk\\
	0 & 0 &-2\qi-2\qj\\
	0&1-\qi-\qj-\qk&-1+\qi
	\end{pmatrix}
	\end{equation*}
	such that  
	\begin{equation*}
P^{-1}AP= \begin{pmatrix}
\qi&0&0\\
0&\qi & 1\\
0&0& \qi
\end{pmatrix}.
\end{equation*}
This implies that the algebraic multiplicity of $\lambda=\qi$ is $2$.
\end{example}

\begin{example}\label{ex3}
	Consider the system $\dot{\bx}=A\bx$, where
	\begin{equation*}
	A= \begin{pmatrix}
	-1+2\qj-\qk&-1+2\qi+\qj\\
	-i+\qj+2\qk&-2-\qi+\qk
	\end{pmatrix}
	\end{equation*}
	The principal fundamental matrix $ M(t) $ at $t=0$ is given by
	\begin{equation*}\label{fundamental matrix 3}
	\begin{pmatrix}
	\frac{3+\qi+\qj-\qk}{6}+\frac{2-\qj-\qk}{6}\gamma_1+\frac{1-\qi+2\qk}{6}\gamma_2&  \frac{-1+3\qi-\qj-\qk}{6}+\frac{2\qi+\qj-\qk}{6}\gamma_1+\frac{1-\qi+2\qk}{6}\gamma_2 \\
	\frac{-1-3\qi+\qj+\qk}{6}+\frac{2\qi+\qj-\qk}{6}\gamma_1+\frac{1+\qi-2\qj}{6}\gamma_2&  \frac{3- \qi-\qj-\qk}{6}+\frac{2 +\qj+\qk}{6}\gamma_1+\frac{1+\qi-2\qj}{6}\gamma_2
	\end{pmatrix}
	\end{equation*}
	where $\gamma_1=e^{-(3+3\qi)t}$ and $\gamma_2=e^{(3\qi-3)t}$.
	By straightforward computations, we have the   result shown in Table \ref{table-ex-3}.
	\begin{table}[ht]
		\centering
		\caption{Description of Example \ref{ex3}}{
			\begin{tabular}{| @{}c | c | c@{} |}
				\hline
				~Fundamental &  The standard  & \multirow{2}*{\centering Stability}\\
				matrix  & eigenvalues of $A$   &    \\ \hline
				$\displaystyle \norm{M(t)}$   &  $\lambda_1=0$, $\Re(\lambda_1)=0$; & stable but not   \\
				is bounded      &  $ \lambda_2 =-3+3\qi$, $\Re(\lambda_1)<0$ &  asymptotically~ \\  \hline
		\end{tabular}}
		\label{table-ex-3}
	\end{table}
\end{example}

\begin{example}\label{ex4}
	Consider the system $\dot{\bx}=A\bx$, where
	\begin{equation*}
	A= \begin{pmatrix}
	-1+\qi-\qk&-\qi\\
	1+\qi-\qj+\qk&-2-\qk
	\end{pmatrix}
	\end{equation*}
	The principal fundamental matrix $ M(t) $  at $t=0$ is given by
	\begin{equation*}\label{fundamental matrix 4}
	\begin{pmatrix}
	(\frac{1-\qi}{2}e^{\qi t}+\frac{\qk-\qj}{2}e^{-\qi t})e^{-2t} + \frac{1+\qi+\qj-\qk}{2}e^{-t}& \frac{1-\qi}{2}e^{(\qi-2)t}+\frac{1-\qi}{2}e^{-t} \\
	(1+\qi)(1-e^{(\qj-1)t})e^{-t} &  \frac{1-\qi-\qj-\qk}{2}e^{-t}+\frac{1+\qi+\qj+\qk}{2}e^{(\qi-2)t}
	\end{pmatrix}
	\end{equation*}
	By straightforward computations,  we have the   result shown in Table  \ref{table-ex-4}.
	\begin{table}[ht]
		\centering
		\caption{Description of Example \ref{ex4}}{
			\begin{tabular}{ |@{} c | c | c@{} |}
				\hline
				Fundamental &  The standard  & \multirow{2}*{\centering Stability}\\
				matrix  & eigenvalues of $A$   &    \\ \hline
				\multirow{2}*{\centering $~\displaystyle\lim_{t\to \infty}\norm{M(t)}=0$ }   &  $\lambda_1=-1$, $\Re(\lambda_1)<0$; &   asymptotically~  \\
				&  $ \lambda_2 =-1+ \frac{\qi}{2}$, $\Re(\lambda_1)<0$ & stable  \\  \hline
		\end{tabular}}
		\label{table-ex-4}
	\end{table}
\end{example}

\section{Floquet theory for QDEs}\label{S4}

We consider  the  quaternionic periodic systems
\begin{equation}\label{q-periodic-systems}
\dot{\bx}=A(t) \bx
\end{equation}
where $A(t)$ is a $T$-periodic continuous quaternion-matrix-valued function. The following Floquet's theorem gives a canonical form for fundamental matrices of (\ref{q-periodic-systems}).

\begin{theorem}
	If $M(t)$ is a fundamental matrix of  (\ref{q-periodic-systems}). Then
	\begin{equation*}
	M(t+T)=M(t)M^{-1}(0)M(T).
	\end{equation*}
	In addition, it has the form
	\begin{equation}\label{floquetnormalform}
	M(t)=P(t)e^{tB}
	\end{equation}
	where P(t) is a $T$-periodic quaternion-matrix-valued function and $B$ satisfying
	\begin{equation*}
	e^{TB}=M^{-1}(0)M(T).
	\end{equation*}
\end{theorem}
\begin{proof}
	Since $M(t)$ is a fundamental matrix of (\ref{q-periodic-systems}) and $A(t+T)=A(t)$, then
	\begin{equation*}
	\dot{M}(t+T)=A(t+T)M(t+T)=A(t)M(t+T).
	\end{equation*}
	That means $M(t+T)$ is also a fundamental matrix. Therefore, there is a nonsingular   quaternion matrix  $C$ such that $M(t+T)=M(t)C$. By Lemma  \ref{exp q-matrix eqn}, there is a quaternion matrix $B$ such that $C=e^{TB}$. Let $P(t):=M(t)e^{-tB}$, then
	\begin{equation*}
	P(t+T)=M(t+T)e^{-TB-tB}=M(t)C e^{TB}e^{-tB}=M(t)e^{-tB}=P(t)
	\end{equation*}
	and $M(t)=P(t)e^{tB}$. By letting $t=0$, we have  $ e^{TB}=C=M^{-1}(0)M(T)$ which completes the proof.
\end{proof}
\begin{remark}
	 In the above proof, we used the fact that if $A_1, A_2 \in \H^{n\times n}$ are commutable then $e^{A_1}e^{A_2}=e^{A_1+A_2}$. We know that this assertion  is true for complex matrices.  We now verify that this result is also valid for quaternion matrices. 
	 
	 If $A_1, A_2$ are commutable, so are $\chi_{A_1}$ $\chi_{A_2}$. By applying Lemma \ref{expmultiply} and \ref{exp q-matrix eqn}, we have that
	 \begin{equation*}
	 \chi_{e^{A_1}e^{A_2}}=\chi_{e^{A_1}}\chi_{e^{A_2}}=e^{\chi_{A_1}}e^{\chi_{A_2}}=e^{\chi_{A_1}+\chi_{A_2}}=e^{\chi_{(A_1+A_2)}}=\chi_{e^{A_1+A_2}}.
	 \end{equation*}
	 It follows that $e^{A_1}e^{A_2}=e^{A_1+A_2}$.
\end{remark}

\begin{corollary}\label{monodromy matrix}
	Suppose that $M_1(t)$, $M_2(t)$ are  fundamental   matrices of (\ref{q-periodic-systems}) and $e^{TB_1}=M_1^{-1}(0)M_1(T)$, $e^{TB_2}=M_2^{-1}(0)M_2(T)$. Then $e^{TB_1}$, $e^{TB_2}$ are similar and therefore they have the same standard eigenvalues.
\end{corollary}
\begin{proof}
	Let  $M_0(t)$ be  the fundamental matrix such that $M_0(0)=I$, then $M_1(t)=M_0(t)M_1(0)$ and $M_2(t)=M_0(t)M_2(0)$ for every $t\in \mathbb{R}$. Therefore $M_1(T)M_1^{-1}(0)=M_2(T)M_2^{-1}(0)=M_0(T)$. Note that both $M_1^{-1}(0)M_1(T)$  and $M_2^{-1}(0)M_2(T)$ are similar with $M_0(T)$. Thus  $e^{TB_1}$, $e^{TB_2}$ are similar and they possess the same standard eigenvalues.
\end{proof}

The representation (\ref{floquetnormalform}) is called a Floquet normal form for the fundamental matrix $M(t)$. From this normal form, we   accordingly  define several concepts for  quaternionic periodic system (\ref{q-periodic-systems}) as follows.
\begin{itemize}
	\item For any  fundamental matrix $M(t)$, $ e^{TB}=M^{-1}(0)M(T)$ is called a monodromy matrix of  (\ref{q-periodic-systems}). By Corollary \ref{monodromy matrix}, we see that any two monodromy matrices are similar.
	\item  The standard eigenvalues of  any monodromy matrix are called characteristic multipliers of (\ref{q-periodic-systems}). The totality of characteristic multipliers is denoted by ${CM}$.
	\item A complex number $\mu$ is called a characteristic exponent of (\ref{q-periodic-systems}), if $\rho$ is a characteristic multiplier and $e^{\mu T}=\rho$. The totality of characteristic exponents is denoted by ${CE}$.
\end{itemize}

\begin{theorem}\label{eigenvalue of B}
	Consider system (\ref{q-periodic-systems}),  suppose that $M(t)=P(t)e^{tB}$ is a Floquet norm form for the fundamental matrix $M(t)$. Let  $\mu_1,\mu_2,\cdots,\mu_n$ be the standard eigenvalues of $B$. Then  $\widetilde{\mu}_1,\widetilde{\mu}_1,\cdots,\widetilde{\mu}_n$  are characteristic exponents, where $\widetilde{\mu}_j$ ($j=1,2,\cdots,n$) is defined by
	\begin{equation*}
	\widetilde{\mu}_j:=
	\begin{cases}
	\mu_j,& \mathrm{if}~ e^{\mu_j T} ~\mathrm{has~ nonnegative ~imaginary~ part};\\
	\overline{\mu_j}, &\mathrm{otherwise}.
	\end{cases}
	\end{equation*}
	If $\mu$ is a characteristic exponent of (\ref{q-periodic-systems}), then there exists  $1\leq k\leq n$ such that  $\{e^{\mu T}\}\cap \{e^{\mu_k T}, e^{ \overline{\mu_k} T}\}\neq \emptyset$ and  $\Re(\mu)=\Re(\mu_k)$.
\end{theorem}
\begin{proof}
	If $\mu_1,\mu_2,\cdots,\mu_n$ are the standard eigenvalues of $B$,   from   Theorem \ref{spectral exponetial map}, $e^{\widetilde{\mu}_jT}$ ($j=1,2,\cdots,n$) is a  standard eigenvalue of $e^{TB}$. That is, $\{e^{\widetilde{\mu}_jT} :j=1,2,\cdots,n\}={CM}$.  Therefore  $\widetilde{\mu}_1,\widetilde{\mu}_1,\cdots,\widetilde{\mu}_n$  are characteristic exponents. If $\mu$ is  a characteristic exponent, then $\rho=e^{\mu  T}$ is a standard eigenvalue of $e^{TB}$. Hence there exists $1\leq k\leq n$, such that $\rho=e^{\widetilde{\mu}_k T}$. It follows that
	\begin{equation*}
	\{e^{\mu T}\}\cap \{e^{\mu_k T}, e^{ \overline{\mu_k} T}\}\neq \emptyset,
	\end{equation*}
	and
	$$e^{\Re(\mu)T}=\abs{e^{\mu  T}}=\abs{e^{\widetilde{\mu}_k T}}=e^{\Re(\widetilde{\mu}_k)T}=e^{\Re({\mu}_k)T}.$$
	Thus  $\Re(\mu)=\Re(\mu_k)$.
\end{proof}

As an immediate consequence of Theorem \ref{eigenvalue of B}, we  have the following result.

\begin{corollary}
	Consider  system (\ref{q-periodic-systems}), Let  $M(t)=P(t)e^{tB}$ be a Floquet norm form for the fundamental matrix $M(t)$. Then
	\begin{equation*}
	\{ \Re(\mu): \mu \in {CE}\}=\{ \Re(\mu): \mu \in \sigma(B)\}
	\end{equation*}
	where $\sigma(B)$ is the totality of the standard eigenvalues of $B$.
\end{corollary}

\begin{theorem}
	If $\rho_j=e^{\mu_j T}$, $j=1,2,\cdots,n$, are the  characteristic multipliers of (\ref{q-periodic-systems}), then
	\begin{align}
	&\prod_{j=1}^{n} \abs{\rho_j}= \exp\left(\int_0^T \Re(\mathrm{tr} A(\tau))d\tau\right),  \label{characteristicmultipliers} \\
	& \Re\left(\sum_{j=1}^n\mu_j\right)=
	\frac{1}{T}\left(\int_0^T\Re\left( \mathrm{tr}A(\tau)\right)d\tau\right). \label{characteristicexponent}
	\end{align}
\end{theorem}
\begin{proof}
	Let $M(t)$ be a fundamental matrix of (\ref{q-periodic-systems}), by  Liouville's formula of QDEs (see \cite{kou2015linear2}), we have
	\begin{equation}\label{Liouville}
	\abs{M(t)}_q=\exp\left(2\int_{t_0}^t \Re(\mathrm{tr} A(\tau))d\tau\right)\abs{M(t_0)}_q.
	\end{equation}
	Note that $\rho_j$, $j=1,2,\cdots,n$, are the standard eigenvalues of $M(T)M^{-1}(0)$, by the definition of $q$-determinant, we have
	\begin{equation*}
	\abs{M(T)}_q\abs{M(0)}_q^{-1}=\abs{M(T)M^{-1}(0)}_q=\prod_{j=1}^{n} \abs{\rho_j}^2.
	\end{equation*}
	Taking $t=T, t_0=0$ in (\ref{Liouville}), we obtain
	\begin{equation*}
	\prod_{j=1}^{n} \abs{\rho_j}^2=  \exp\left(2\int_{0}^T \Re(\mathrm{tr} A(\tau))d\tau\right),
	\end{equation*}
	and therefore (\ref{characteristicmultipliers}) holds. Observe that $\abs{\rho_j}=\abs{e^{\mu_j T}}=e^{\Re(\mu_j) T}$, then (\ref{characteristicmultipliers}) implies that
	\begin{equation*}
	\exp\left(\Re\left(\sum_{j=1}^n\mu_j\right)T\right)=\exp\left(\int_0^T  \Re(\mathrm{tr} A(\tau))d\tau\right).
	\end{equation*}
	This proves the theorem.
\end{proof}

If $\rho=e^{\mu T}$, where $\rho, \mu$ are complex numbers. Since $\abs{\rho}=\abs{e^{\mu T}}=e^{\Re(\mu) T}$, it is easy to see that the following assertions hold.
\begin{itemize}
	\item $\abs{\rho}=1$ if and only if $\Re(\mu)=0$.
	\item $\abs{\rho}<1$ if and only if $\Re(\mu)<0$.
	\item $\abs{\rho}>1$ if and only if $\Re(\mu)>0$.
\end{itemize}

The next result demonstrates that the stability of  (\ref{q-periodic-systems}) is equivalent to the stability of the linear system with constant coefficients $\dot{\by}=B \by$, where $B$ stems from the Floquet normal form (\ref{floquetnormalform}).

\begin{theorem}\label{stablility of quaternionic periodic system}
	Let   $M(t)=P(t)e^{tB}$ is a Floquet norm form for the fundamental matrix $M(t)$ of (\ref{q-periodic-systems}). Then the following assertions hold.
	\begin{enumerate}
		\item The system (\ref{q-periodic-systems}) is stable if and only if the standard eigenvalues of $B$ all have non-positive real parts and  the algebraic multiplicity equals the geometric multiplicity of each standard eigenvalue with   zero real part; or equivalently, the characteristic multipliers of (\ref{q-periodic-systems}) all have modulus not larger than $1$ ($\leq1$) and  the algebraic multiplicity equals the geometric multiplicity of each characteristic multiplier with    modulus one.
		\item  The system (\ref{q-periodic-systems}) is  asymptotically stable if and only if the standard eigenvalues of $B$ all have negative  real parts; or equivalently, the characteristic multipliers of (\ref{q-periodic-systems}) all have modulus less than $1$.
	\end{enumerate}
\end{theorem}

\begin{theorem}\label{periodic sol}
	If $\mu$ is a characteristic exponent and $\rho=e^{\mu T}$ is a characteristic multiplier of (\ref{q-periodic-systems}), then there is a nontrivial solution of the form
	\begin{equation*}
	\bx(t)=\bp (t)e^{\mu t}.
	\end{equation*}
	Moreover  $\bp(t+T)=\bp(t)$ and $\bx(t+T)=\bx(t)\rho$.
\end{theorem}
\begin{proof}
	Let   $M(t)=P(t)e^{tB}$ is a Floquet norm form  for the principal fundamental matrix $M(t)$ at $t=0$. By Theorem \ref{eigenvalue of B}, there is a standard eigenvalue $\mu_1$ of $B$ such that $$\{e^{\mu T}\}\cap \{e^{\mu_1 T}, e^{ \overline{\mu_1} T}\}\neq \emptyset.$$
	Without loss of generality, we assume that $\rho=e^{\mu T}=e^{\mu_1 T}$. Then there exists a $k\in\mathbb{Z}$ such that  $\mu_1=\mu+ \frac{2k\pi\qi}{T}$. Let $\bmeta\neq 0$ be an eigenvector of $B$ corresponding to $\mu_1$. It follows that $B\bmeta=\bmeta \mu_1$ and therefore $e^{tB}\bmeta=\bmeta e^{\mu_1 t}$. Thus the solution $\bx(t):=M(t)\bmeta$ can also be represented in the form
	\begin{equation*}
	\bx(t)=P(t)e^{tB}\bmeta=P(t)\bmeta e^{\frac{2k\pi\qi t}{T}}e^{\mu t}.
	\end{equation*}
	Let $\bp(t)=P(t)\bmeta e^{\frac{2k\pi\qi t}{T}}$. It is easy to see that
	$\bp(t)$ is a $T$-periodic function. Moreover
	\begin{equation*}
	\bx(t+T)=\bp(t+T)e^{\mu(t+T)}=\bp(t)e^{\mu t}e^{\mu T}=\bx(t) \rho.
	\end{equation*}
	This completes the proof.
\end{proof}
\begin{theorem}\label{periodic sol inverse}
	If $\mu$ is a complex number,  $\bp(t+T)=\bp(t)$, and $\bx(t)=\bp(t)e^{\mu t}\neq 0$    is a nontrivial solution of  (\ref{q-periodic-systems}), then one of $\mu, \overline{\mu}$ is a characteristic exponent.
\end{theorem}
\begin{proof}
	Let   $M(t)=P(t)e^{tB}$ be a Floquet norm form  for the principal fundamental matrix $M(t)$ at $t=0$ and $\bmeta=\bp(0)$, then $\bmeta\neq 0$. Otherwise, $\bx(t)\equiv 0$ is the trivial solution by  uniqueness of solution. Note that both $\bp(t)e^{\mu t}$ and $P(t)e^{t B}\bmeta$ are solutions of (\ref{q-periodic-systems}) with the same initial value at $t=0$, therefore
	\begin{equation}\label{same-initial}
	\bp(t)e^{\mu t}=P(t)e^{t B}\bmeta
	\end{equation}
	Taking $t=T$ in (\ref{same-initial}) and note that
	$\bp(T)=\bp(0)=\bmeta$, $P(T)=P(0)=I$    by periodicity. It follows that
	\begin{equation*}
	\bmeta e^{\mu T}=e^{TB}\bmeta.
	\end{equation*}
	Hence $e^{\mu T}$ is a complex-valued eigenvalue of $e^{TB}$. Thus, one of $e^{\mu T}$, $e^{\overline{\mu} T}$ is a  characteristic multiplier of (\ref{q-periodic-systems}). Therefore, one of $\mu, \overline{\mu}$ is a characteristic exponent of (\ref{q-periodic-systems}).
\end{proof}
%%%%%%%%%%%%%%%%%%%%%%%%%%%%%%%%%%%%%%%%%%%%%%%%%%%%%%%%%%%%%%%%%%%%%%%%%%%%%%%%%%%%%%%%%%%%%%%%%%%%%%

Next result is  a direct consequence of Theorem  \ref{periodic sol} and \ref{periodic sol inverse}.
\begin{corollary}\label{T-periodic-or-2T-periodic}
	There is a $T$-periodic solution of (\ref{q-periodic-systems})
	if and only if there is a zero characteristic exponent; or equivalently, there is a characteristic  multiplier  $\rho=1$. If there is a  characteristic exponent of the form $\mu=\frac{2k+1}{T}\pi \qi$ for some $k\in \mathbb{Z}$, or equivalently, there is a characteristic  multiplier  $\rho=-1$, then  there is a $2T$-periodic solution of (\ref{q-periodic-systems}).
\end{corollary}

The following result shows that different characteristic  multipliers will generate linearly independent solutions.
\begin{corollary}
	Assume that  $\mu_1, \mu_2$ are characteristic exponents of (\ref{q-periodic-systems}) satisfying $\rho_1=e^{\mu_1 T}$, $\rho_2=e^{\mu_2 T}$. If the characteristic  multipliers $\rho_1$, $\rho_2$ are not equal, then there are $T$-periodic functions $\bp_1(t)$, $\bp_2(t)$ such that
	\begin{equation*}
	\bx_1(t)=\bp_1(t)e^{\mu_1 t}
	\end{equation*}
	and
	\begin{equation*}
	\bx_2(t)=\bp_2(t)e^{\mu_2 t}
	\end{equation*}
	are linearly independent solutions of (\ref{q-periodic-systems}).
\end{corollary}
\begin{proof}
	Let   $M(t)=P(t)e^{tB}$ be a Floquet norm form  for the principal fundamental matrix $M(t)$ at $t=0$ and $\bmeta_1=\bx_1(0)$, $\bmeta_2=\bx_2(0)$. By similar arguments of Theorem
	\ref{periodic sol inverse}, we conclude that $\bmeta_1$, $\bmeta_2$ are eigenvectors of $B$ corresponding to the standard eigenvalues $\rho_1$, $\rho_2$ respectively. Note that $\rho_1\neq\rho_2$. It follows that $\bx_1(0)$ and $\bx_2(0)$ are linearly independent and therefore $\bx_1(t)$ and $\bx_2(t)$  are linearly independent solutions of (\ref{q-periodic-systems}).
\end{proof}

\begin{example}\label{ex5}
	Consider the system  (\ref{q-periodic-systems}), where $A(t)$ is $\pi$-periodic function and given by
	\begin{equation*}
	A(t)= \begin{pmatrix}
	1&1\\
	0 &\qi+2e^{2\qi t}\qj
	\end{pmatrix}
	\end{equation*}
	Then the principal fundamental matrix is
	\begin{equation*}
	M(t)= \begin{pmatrix}
	e^t& \frac{-1+\qi-\qj-\qk}{4}e^{\qj t}+\frac{-1-3\qi-3\qj+\qk}{20}e^{3\qj t}+\frac{3-\qi+4\qj+2\qk}{10}e^t\\
	0 &e^{\qi t}e^{2\qj t}
	\end{pmatrix}.
	\end{equation*}
	By straightforward computations, we have $\lim_{t\to \infty}\norm{M(t)}=\infty$. That is, $\norm{M(t)}$ is unbounded. Thus this system is unstable by Theorem \ref{judging theorem}.
	Observe that $M(0)=I$ and
	\begin{equation*}
	M(\pi)= \begin{pmatrix}
	e^{\pi}& \frac{3-\qi+4\qj+2\qk}{10}(1+e^{\pi})\\
	0 &-1
	\end{pmatrix}.
	\end{equation*}
	Therefore the characteristic multipliers are $\rho_1=e^{\pi}$, $\rho_2=-1$.
	From Lemma \ref{exp q-matrix eqn}, there is a quaternion-valued matrix
	\begin{equation*}
	B= \begin{pmatrix}
	1&  \frac{1-2\qi+\qj+ 3\qk}{5}\\
	0 &\qi
	\end{pmatrix}
	\end{equation*}
	such that  $M(\pi)=e^{\pi B}$. Applying the definition of exponential function,
	\begin{equation*}
	e^{tB}= \begin{pmatrix}
	e^t&  \frac{3- \qi+ 4 \qj+ 2\qk}{10} (e^t-e^{\qi t})\\
	0 &e^{\qi t}
	\end{pmatrix}
	\end{equation*}
	Then we obtain  the Floquet norm form $P(t)e^{tB}$  for   $M(t)$, where  $P(t)$ is given by
	\begin{equation*}
	P(t)= \begin{pmatrix}
	1&  \frac{3- \qi+ 4 \qj+ 2\qk}{10} +\frac{-1+\qi-\qj-\qk}{4}e^{\qj t}e^{-\qi t}+\frac{-1-3\qi-3\qj+\qk}{20} e^{3\qj t}e^{-\qi t}\\
	0 &\cos 2t +e^{2\qi t}\qj\sin 2t
	\end{pmatrix}
	\end{equation*}
	It is easy to see that $P(t)$ is $\pi$-periodic as required. The standard eigenvalues of $B$ are $\mu_1=1$, $\mu_2=\qi$ and the corresponding eigenvectors are
	\begin{equation*}
	\bmeta_1= \begin{pmatrix}
	1 \\
	0
	\end{pmatrix}~~~\mathrm{and}~~~ \bmeta_2= \begin{pmatrix}
	-\frac{7+\qi+10\qj}{10} \\
	2+\qi
	\end{pmatrix}.
	\end{equation*}
	Note that $\mu_1, \mu_2$ are characteristic exponents. By Theorem \ref{periodic sol}, there are two nontrivial solutions
	\begin{equation*}
	\bx_1(t)=
	M(t)\bmeta_1=\bp_1(t)e^t
	~~\mathrm{and}~~ \bx_2(t)=M(t)\bmeta_2=\bp_2(t)e^{\qi t},
	\end{equation*}
	where $\bp_1(t), \bp_2(t)$ are $\pi$-periodic functions given by
	\begin{equation*}
	\bp_1(t)= \begin{pmatrix}
	1 \\
	0
	\end{pmatrix}~~\mathrm{and}~~ \bp_2(t)= \begin{pmatrix}
	\frac{-1+\qi-\qj-\qk}{4}e^{\qj t}e^{-\qi t}(2+\qi)+\frac{-1-3\qi-3\qj+\qk}{20} e^{3\qj t}e^{-\qi t}(2+\qi) \\
	(2+\qi)\cos2t+2 e^{2\qi t}(\qj+\qk)\sin2t
	\end{pmatrix}.
	\end{equation*}
	By Corollary \ref{T-periodic-or-2T-periodic}, $\bx_2(t)$ is a $2\pi$-periodic solution. To provide a direct description of the system, Table \ref{table-ex-5}  is presented to visualize its properties.
	\begin{table}[ht]
		\centering
		\caption{Description of Example \ref{ex5}}{
			\begin{tabular}{| @{} c | c | c| c@{}|}
				\hline
				~Fundamental & Characteristic&  The standard  & \multirow{2}*{\centering Stability~}\\
				matrix  &  multipliers & eigenvalues of $B$   &    \\ \hline
				$\norm{M(t)}$ is  &  $\rho_1=e^{\pi}$, $\abs{\rho_1}>1$;&  $\mu_1=1$, $\Re(\mu_1)>0$; & \multirow{2}*{\centering unstable~} \\
				unbounded  &$\rho_2=-1$, $\abs{\rho_2}=1$&  $\mu_2=\qi$, $\Re(\mu_2)=0$ &  \\ \hline
		\end{tabular}}
		\label{table-ex-5}
	\end{table}
\end{example}

\begin{example}\label{ex6}
	Consider the system  (\ref{q-periodic-systems}), where $A(t)$ is $\pi$-periodic function and is given by
	\begin{equation*}
	A(t)= \begin{pmatrix}
	\qk&1\\
	0 &\qi+2e^{2\qi t}\qj
	\end{pmatrix}
	\end{equation*}
	Then the principal fundamental matrix $M(t)$ is
	\begin{equation*}
	\begin{pmatrix}
	e^{\qk t}& \frac{1-\qi-\qj+\qk}{4}\sin t +\frac{1+\qi+\qj+\qk}{4}e^{ \qj t} t +\frac{1+\qi-\qj-\qk}{4}e^{2\qj t}\sin t + \frac{1-\qi-\qj+\qk}{16}(e^{3\qj t}-e^{-\qj t})\\
	0 &e^{\qi t}e^{2\qj t}
	\end{pmatrix}.
	\end{equation*}
	By straightforward computations, $\norm{M(t)}$ is unbounded. Thus this system is unstable by Theorem \ref{judging theorem}.
	Observe that $M(0)=I$ and
	\begin{equation*}
	M(\pi)= \begin{pmatrix}
	-1& - \frac{1+\qi+ \qj+ \qk}{4} \pi\\
	0 &-1
	\end{pmatrix}.
	\end{equation*}
	Therefore the characteristic multipliers are $\rho_1= \rho_2=-1$.
	There is a quaternion-valued matrix
	\begin{equation*}
	B= \begin{pmatrix}
	-\qi& \frac{1+\qi+ \qj+ \qk}{4}\\
	0 &-\qk
	\end{pmatrix}
	\end{equation*}
	such that  $M(\pi)=e^{\pi B}$. The standard eigenvalues of $B$ are $\mu_1=  \mu_2=\qi$. To provide a direct description of the system, Table \ref{table-ex-6} is presented to visualize its properties. By some basic calculations, we obtain the Jordan canonical form of $B$:
	\begin{equation*}
	J=\begin{pmatrix}
	\qi& 1\\
	0 &\qi
	\end{pmatrix}.
	\end{equation*}
	This implies that the geometric multiplicity of $\mu=\qi$ is $1$, which is less than its algebraic multiplicity.
	\begin{table}[ht]
		\centering
		\caption{Description of Example \ref{ex6}}{
			\begin{tabular}{ |@{} c | c | c| c @{}|}
				\hline
				~Fundamental & Characteristic&  The standard  & \multirow{2}*{\centering Stability~}\\
				matrix  &  multipliers & eigenvalues of $B$   &    \\ \hline
				$\norm{M(t)}$ is  &  $\rho_1=\rho_2=-1$;&  $\mu_1=\mu_2=\qi$; & \multirow{2}*{\centering unstable~} \\
				unbounded  &$\abs{\rho_1}=\abs{\rho_2}=1 $ &   $\Re(\mu_1)=\Re(\mu_2)=0$ &  \\ \hline
		\end{tabular}}
		\label{table-ex-6}
	\end{table}
\end{example}

\begin{example}\label{ex7}
	Consider the system  (\ref{q-periodic-systems}), where $A(t)$ is $\pi$-periodic function and given by
	\begin{equation*}
	A(t)= \begin{pmatrix}
	\frac{\qk}{2}&e^{-2\qi t}\\
	0 &\qi+2\qj\cos 2t+2\qk\sin 2t
	\end{pmatrix}
	\end{equation*}
	Then the principal fundamental matrix $M(t)$ is
	\begin{equation*}
	\begin{pmatrix}
	e^{\frac{\qk}{2} t}&  \frac{-2+2\qi+5\qj-5\qk}{21}e^{-\frac{\qj}{2}t} +\frac{2+2\qi+3\qj+3\qk}{5}e^{ \frac{\qj}{2} t}   -\frac{1+2\qi+2\qj-\qk}{4}e^{\qj t} + \frac{1+6\qi-6\qj-\qk}{35} e^{3\qj t} \\
	0 &e^{\qi t}e^{2\qj t}
	\end{pmatrix}.
	\end{equation*}
	It is easy to see that $\norm{M(t)}$ is  bounded but is not convergent to zero as $t$ tends to infinity.  Thus this system is  stable (but not asymptotically) by Theorem \ref{judging theorem}.
	Observe that $M(0)=I$ and
	\begin{equation*}
	M(\pi)= \begin{pmatrix}
	\qk&  -\frac{2}{35}-\frac{12}{35}\qi + \frac{4}{3}\qj + \frac{2}{3}\qk\\
	0 &-1
	\end{pmatrix}.
	\end{equation*}
	Therefore the characteristic multipliers are $\rho_1=\qi$, $\rho_2=-1$.
	There is a quaternion-valued matrix
	\begin{equation*}
	B= \begin{pmatrix}
	\frac{\qk}{2}& \frac{33-76\qi-12\qj+104\qk }{105}\\
	0 &\qi
	\end{pmatrix}
	\end{equation*}
	such that  $M(\pi)=e^{\pi B}$. The standard eigenvalues of $B$ are $\mu_1= \frac{\qi}{2}$, $\mu_2=\qi$. To provide a direct description of the system, Table \ref{table-ex-7} is presented to visualize its properties.
	\begin{table}[ht]
		\centering
		\caption{Description of Example \ref{ex7}}{
			\begin{tabular}{| @{} c | c | c| c @{}|}
				\hline
				Fundamental & Characteristic&  The standard  & \multirow{2}*{\centering Stability}\\
				matrix  &  multipliers & eigenvalues of $B$   &    \\ \hline
				$\norm{M(t)}$ is unbounded  &  $\rho_1=\qi\neq -1= \rho_2$;&  $\mu_1= \frac{\qi}{2} \neq \qi=\mu_2$; & stable but not \\
				~but not convergent to $0$ &$\abs{\rho_1}=\abs{\rho_2}=1$&  $\Re(\mu_1)= \Re(\mu_2)=0$ &   asymptotically~ \\ \hline
		\end{tabular}}\label{table-ex-7}
	\end{table}
\end{example}

\begin{example}\label{ex8}
	Consider the system  (\ref{q-periodic-systems}), where $A(t)$ is $\pi$-periodic function and given by
	\begin{equation*}
	A(t)= \begin{pmatrix}
	\frac{\qi}{2}-1 &e^{2\qj t}e^{-\qk \sin 2t}\\
	0 &2\qk \cos 2t -1
	\end{pmatrix}
	\end{equation*}
	Then the principal fundamental matrix $M(t)$ is
	\begin{equation*}
	\begin{pmatrix}
	e^{\frac{\qi}{2} t}e^{-t}&   \frac{ 1}{5}(e^{-(1+2\qi)t} - e^{(\frac{\qi}{2}-1)t})(\qi-\qj) + \frac{1 }{3}(e^{(\frac{\qi}{2}-1)t}-e^{(2\qi-1)t})(\qi+\qj) \\
	0 & e^{-t}e^{\qk \sin 2t}
	\end{pmatrix}.
	\end{equation*}
	It is easy to see that $\lim_{t\to \infty}\norm{M(t)}=0$.  Thus this system is asymptotically stable  by Theorem \ref{judging theorem}.
	Observe that $M(0)=I$ and
	\begin{equation*}
	M(\pi)= \begin{pmatrix}
	\qi e^{-\pi}&  \frac{-2-2\qi-8\qj+8\qk}{15}e^{-\pi}\\
	0 &  e^{-\pi}
	\end{pmatrix}.
	\end{equation*}
	Therefore the characteristic multipliers are $\rho_1=\qi e^{-\pi}$, $\rho_2= e^{-\pi}$.
	There is a quaternion-valued matrix
	\begin{equation*}
	B= \begin{pmatrix}
	\frac{\qi}{2}-1& \frac{-1+4\qk }{15}\\
	0 & -1
	\end{pmatrix}
	\end{equation*}
	such that  $M(\pi)=e^{\pi B}$. The standard eigenvalues of $B$ are $\mu_1= \frac{\qi}{2}-1$, $\mu_2=-1$. To provide a direct description of the system, Table \ref{table-ex-8} is presented to visualize its properties.
	\begin{table}[ht]
		\centering
		\caption{Description of Example \ref{ex8}}{
			\begin{tabular}{ |@{} c | c | c| c @{}|}
				\hline
				Fundamental & Characteristic&  The standard  & \multirow{2}*{\centering Stability}\\
				matrix  &  multipliers & eigenvalues of $B$   &    \\ \hline
				\multirow{2}*{\centering $~\displaystyle\lim_{t\to \infty}\norm{M(t)}=0$ }&  $\rho_1=\qi e^{-\pi}$, $\abs{\rho_1}<1$;&  $\mu_1=\frac{\qi}{2}-1$, $\Re(\mu_1)<0$; & asymptotically~ \\
				&$\rho_2=e^{-\pi}$, $\abs{\rho_2}<1$&  $\mu_2=-1$, $\Re(\mu_2)<0$ &  stable\\ \hline
		\end{tabular}}\label{table-ex-8}
	\end{table}
\end{example}

\begin{remark}
	Thanks to the assertion 2 of Theorem \ref{thm of q matrix}, the above results  are coincide with the   traditional results when $A(t)$ is complex-valued.
\end{remark}

\section{Quaternion-valued Hill's equations}\label{S5}

For real-valued  systems, the Floquet theory effectively depict the stability of Hill's equation (see e.g \cite{chicone2006ordinary})
\begin{equation*}\label{Hill eqn}
\ddot{u} +a(t)u=0,~~a(t)=a(t+T).
\end{equation*}
We will consider the quaternion case where $a(t)$  is a quaternion-valued function. Let $\bx=(u,u')^T$, then  quaternion-valued Hill's equation is equivalent to the quaternionic periodic systems
(\ref{q-periodic-systems}) with
\begin{equation*}
A(t)=\begin{pmatrix}
0& 1\\
-a(t) &0
\end{pmatrix}.
\end{equation*}

Let $M(t)$ be the principal fundamental matrix at $t=0$.  By  Liouville's formula of QDEs,  we have
\begin{equation*}
\abs{M(t)}_q=\exp\left(2\int_{t_0}^t \Re(\mathrm{tr} A(\tau))d\tau\right)\abs{M(0)}_q=1.
\end{equation*}

If $a(t)$ is real-valued, then $M(T)$ is a real-valued matrix. If $\alpha=\alpha_1+\qi \alpha_2$ and $\beta=\beta_1+\qi \beta_2$ are roots of the equation
\begin{equation}\label{character eqn}
\lambda^2-(\mathrm{tr}M(T))\lambda+\abs{M(T)}_q= \lambda^2-(\mathrm{tr}M(T))\lambda+1=0.
\end{equation}
Then $\rho_1=\alpha_1+\qi \abs{\alpha_2}$ and $\rho_2=\beta_1+\qi \abs{\beta_2}$ are characteristic multipliers of (\ref{q-periodic-systems}) and $\abs{\rho_1} =\abs{\alpha}$, $\abs{\rho_2} =\abs{\beta}$. It is well-known that the stability of real-valued Hill's equation depends on the value of $\mathrm{tr}M(T)$ (see e.g \cite{chicone2006ordinary}).
\begin{table}[ht]
	\centering
	\caption{Description of Real-valued Hill's equation}{
		\begin{tabular}{| @{} c | c | c @{} |}
			\hline
			The value  &The roots &    Stability  of real-valued \\
			of   $\mathrm{tr}M(T)$ &  of (\ref{character eqn})   &   Hill's equation  \\ \hline
			$\mathrm{tr}M(T)<-2$ &  $\alpha<-1<\beta<0$; & unstable  \\   \hline
			$-2<\mathrm{tr}M(T)<2$  &$\beta=\overline{\alpha}$, $\abs{\alpha}=1$, $\Im(\alpha)\neq 0$; & stable but not asymptotically \\ \hline
			$ \mathrm{tr}M(T)=2$  &$\beta= \alpha=1$; & stable if and only if $M(T)=I$ \\ \hline
			$ \mathrm{tr}M(T)>2$  &   $0<\alpha<1<\beta$; & unstable  \\ \hline
			$ \mathrm{tr}M(T)=-2$  &$\beta= \alpha=-1$; & stable if and only if $M(T)=-I$~ \\ \hline
	\end{tabular}}\label{real Hill}
\end{table}

If $a(t)$ is quaternion-valued, then $M(T)$ is a quaternion matrix. Therefore   we can not use (\ref{character eqn}) to find the characteristic multipliers (the standard eigenvalues of $M(T)$). In this case, $\mathrm{tr}M(T)$ is a quaternion.  The structure of the set of zeros of quaternionic polynomials is more complicated than complex polynomials. It is natural  to modify (\ref{character eqn}) to be
\begin{equation}\label{character eqn 2}
\lambda^2-\Re(\mathrm{tr}M(T))\lambda+\abs{M(T)}_q=0.
\end{equation}
This raises the question of whether the roots of (\ref{character eqn 2}) and characteristic multipliers   possess   the same absolute value. The answer is negative. This implies that even if we add $ \Re$ to the front of $\mathrm{tr}M(T)$, the stability of quaternion-valued Hill's equation can not  be determined by Table \ref{real Hill}.
\begin{example}\label{ex-5-1}
	Consider the  quaternion-valued Hill's equation with $a(t)=2+\qj\cos^2 2t+\qk \sin 2t$. Note that $a(t)$ is a quaternion-valued $\pi$-periodic function. Based on the numerical methods, we obtain $M(\pi)\approx\begin{pmatrix} m_1&m_2\\ m_3&m_4 \end{pmatrix}$, where
	\begin{equation*}
	\begin{cases}
	m_1= -0.131186+0.037757\qi+0.584454\qj-0.418119\qk,   \\
	m_2= -0.607206+0.255374\qi-0.025292\qj,  \\
	m_3= 1.900430+0.005637\qi+0.173381\qj,   \\
	m_4= -0.131186+0.037757\qi+0.584454\qj+ 0.418119\qk.
	\end{cases}
	\end{equation*}
	Therefore, by direct computations, we have $ \Re(\mathrm{tr}M(\pi))\approx -0.262372\in (-2,2)$. The  characteristic multipliers are $\rho_1\approx -0.197803+1.73905\qi$ and $\rho_2\approx -0.064569 +0.567682\qi$. Note that $\abs{\rho_1}>1$, thus this equation is unstable. On the other hand, the roots of $\lambda^2- \Re(\mathrm{tr}M(\pi))+\abs{M(\pi)}_q\approx\lambda^2+0.262372\lambda+1=0$ are $\alpha\approx 0.131186  +0.991358 \qi$ and $\beta=\overline{\alpha}$.
\end{example}

In fact, if $\rho_1$, $\rho_2$ are  characteristic multipliers, we only have
\begin{equation}\label{character eqn 3}
\begin{cases}
\Re(\rho_1)+\Re(\rho_2)  =  \Re(\mathrm{tr}M(T)),\\
\abs{\rho_1 } \abs{\rho_2} =\abs{M(T)}_q=1.
\end{cases}
\end{equation}
If $\abs{ \Re(\mathrm{tr}M(T))}>2$, then one of $\abs{\Re(\rho_1)}$ and $\abs{ \Re(\rho_2)}$ has to be larger than $1$. In this case, the equation is unstable. By similar arguments, we could know the stability of quaternion-valued Hill's equation when  $\abs{ \Re(\mathrm{tr}M(T))}=2$.
In summary,    Table \ref{quaternion Hill} is presented to visualize  the  stability of quaternion-valued Hill's equation.
\begin{table}[ht]
	\centering
	\caption{Description of quaternion-valued Hill's equation}{
		\begin{tabular}{ |@{} c  | c @{} |}
			\hline
			The value   &    Stability  of quaternion-valued \\
			of   $ \Re(\mathrm{tr}M(T))$    &   Hill's equation  \\ \hline
			$\abs{ \Re(\mathrm{tr}M(T))}>2$ & unstable  \\   \hline
			$\abs{\Re(\mathrm{tr}M(T))}<2$  & undetermined \\ \hline
			$  \Re(\mathrm{tr}M(T))=2$  & stable if and only if $M(T)=I$ \\ \hline
			$  \Re(\mathrm{tr}M(T))=-2$  & stable if and only if $M(T)=-I$~ \\ \hline
	\end{tabular}}\label{quaternion Hill}
\end{table}

We use the following example to illustrate (\ref{character eqn 3}) and Table \ref{quaternion Hill}.
\begin{example}\label{ex-5-2}
	Consider the  quaternion-valued Hill's equation with $a(t)=-1+\qj\cos  2t+\qk \sin 2t$.  Based on the numerical methods, we obtain $M(\pi)\approx\begin{pmatrix} m_1&m_2\\ m_3&m_4 \end{pmatrix}$, where
	\begin{equation*}
	\begin{cases}
	m_1=13.6488 -2.9075\qi -1.1093\qj -2.3529\qk,   \\
	m_2=12.3192 -2.2187 \qi+2.3529\qj,  \\
	m_3=14.6721-2.2187\qi-5.2605 \qj,   \\
	m_4=13.6488-2.9075\qi -1.1093\qj + 2.3529\qk.
	\end{cases}
	\end{equation*}
	Therefore, by direct computation, we have the following result (Table \ref{table-ex-10}).
	\begin{table}[ht]
		\centering
		\caption{Description of Example \ref{ex-5-2}}{
			\begin{tabular}{ | @{} c | c | c @{} | }
				\hline
				The value  & Characteristic  & \multirow{2}*{\centering Stability~}\\
				of   $ \Re(\mathrm{tr}M(\pi))$  &  multipliers   &    \\ \hline
				\multirow{2}*{\centering $ \Re(\mathrm{tr}M(\pi))\approx 27.2976 >2$ }&  $\rho_1\approx 27.2621 + 4.96756 \qi$, $\abs{\rho_1}>1$;  &  \multirow{2}*{\centering  unstable~}\\
				&$\rho_2 \approx 0.0355 + 0.0065 \qi$, $\abs{\rho_1}\abs{\rho_2}\approx 1$  &  \\ \hline
		\end{tabular}}\label{table-ex-10}
	\end{table}
\end{example}

For the case of $\abs{\Re(\mathrm{tr}M(T))}<2$,  the scalar part  of  $\mathrm{tr}M(T)$ is not enough to determine the stability of quaternion-valued Hill's equation. To take the vector part of $\mathrm{tr}M(T)$ into account, however, we still can't   determine the stability of quaternion-valued Hill's equation at this moment.  This raises the question: can we determine the stability of quaternion-valued Hill's equation by $\mathrm{tr}M(T)$ (including scalar and vector parts)? If yes, how to determine the stability of quaternion-valued Hill's equation by $\mathrm{tr}M(T)$? 

Multiplying $M(T)$ by its conjugate transpose $M(T)^{\dag}$ we construct a positive semidefinite matrix $K(T):=M(T)M(T)^{\dag}$. It is easy to see that the eigenvalues of $K(T)$ are $\kappa_1=\abs{\rho_1}^2, \kappa_2=\abs{\rho_2}^2$. Note that $\abs{\rho_1}\abs{\rho_2} = 1$ and  $\mathrm{tr}K(T)=\norm{M(T)}_F^2$ where $\norm{\cdot}_F$ is the Frobenius norm. It follows that
$\kappa_1, \kappa_2$ are solutions of $\lambda^2- \norm{M(T)}_F^2 \lambda+1=0$.  Then we have the following result.
\begin{theorem}\label{Fnorm}
	 If $\norm{M(T)}_F^2>2$, then quaternion-valued Hill's equation is unstable.
\end{theorem}

By direct computation, we have that $\abs{\Re(\mathrm{tr}M(T))}< \norm{M(T)}_F^2$. It turns out that $\abs{\Re(\mathrm{tr}M(T))}>2$ implies $\norm{M(T)}_F^2>2$. In Example \ref{ex-5-1},  $\abs{\Re(\mathrm{tr}M(T))}=0.262372<2$ and $\norm{M(T)}_F^2=5.14637>2$. It means that the stability of Example \ref{ex-5-1} can be determined by Theorem \ref{Fnorm}. In fact, for some  quaternion-valued Hill's equations with $\abs{\Re(\mathrm{tr}M(T))}<2$, the corresponding $\norm{M(T)}_F^2$ can be very large.
\begin{example}\label{ex-5-4}
	Consider the  quaternion-valued Hill's equation with $a(t)=-1+\qj e^{\cos  2t}+\qk \sin 2t$.   
Based on the numerical methods, we have the following result (Table \ref{table-ex-5-4}).
	\begin{table}[ht]
		\centering
		\caption{Description of Example \ref{ex-5-4}}{
			\begin{tabular}{ | @{} c | c | c @{} | }
				\hline
				The value of & Characteristic  & \multirow{2}*{\centering Stability~}\\
				  ~ $\abs{\Re(\mathrm{tr}M(\pi))}$ and $\norm{M(T)}_F^2$ &  multipliers   &    \\ \hline
				$\abs{\Re(\mathrm{tr}M(\pi))}\approx 1.0394<2$ &  $\rho_1\approx -1.03876 + 40.196  \qi$, $\abs{\rho_1}>1$;  &  \multirow{2}*{\centering  unstable~}\\
				$\norm{M(T)}_F^2\approx  1942.77>2$ 	&$\rho_2 \approx -0.0006425 + 0.024862 \qi$, $\abs{\rho_1}\abs{\rho_2}\approx 1$  &  \\ \hline
		\end{tabular}}\label{table-ex-5-4}
	\end{table}
\end{example}

\section{Conclusions}\label{S6}

The Floquet theory for QDEs  is developed, which  coincides with the classical Floquet theory when considering  ODEs. The concepts of characteristic multipliers and characteristic exponents for QDEs are introduced. The newly obtained results are useful to determine the stability of quaternionic periodic systems. As an important example of applications of Floquet theory for QDEs, we discuss the  stability of  quaternion-valued Hill's equation in detail.
It is shown that some results of real-valued Hill's equation  are  invalid for  the quaternion-valued Hill's equation. Throughout the paper,     adequate examples are provided to support the results.

%\begin{algorithm}[htb]
%	\setstretch{1.35} %设置具有指定弹力的橡皮长度（原行宽的1.35倍）
%	\caption{}
%	\label{alg:Framwork}
%	\begin{algorithmic}
%		\REQUIRE ~~\\
%		The set of positive samples for current batch, $P_n$;\\
%		The set of unlabelled samples for current batch, $U_n$;\\
%		Ensemble of classifiers on former batches, $E_{n-1}$;
%		
%		
%		\ENSURE ~~\\
%		Ensemble of classifiers on the current batch, $E_n$;
%		Extracting the set of reliable negative and/or positive samples $T_n$ from $U_n$ with help of $P_n$;\\
%		\label{ code:fram:extract }
%		
%		Training ensemble of classifiers $E$ on $T_n \cup P_n$, with help of data in former batches;\\
%		\label{code:fram:trainbase}
%		
%		$E_n=E_{n-1}\cup E$;\\
%		\label{code:fram:add}
%		
%		Classifying samples in $U_n-T_n$ by $E_n$;\\
%		\label{code:fram:classify}
%		
%		Deleting some weak classifiers in $E_n$ so as to keep the capacity of $E_n$;\\
%		\label{code:fram:select}
%		
%		$E_n$;\\
%	\end{algorithmic}
%\end{algorithm}

\bibliographystyle{IEEEtran}
\bibliography{IEEEabrv,myreference20181128}

\end{document}